\documentclass{article}

\usepackage[margin=2.5cm]{geometry}
\usepackage{algorithmicx}
\RequirePackage{amsfonts,amssymb,bm,amsthm,amsmath}
\usepackage{graphicx}
\usepackage{subfigure}
\usepackage{booktabs}
\usepackage{amsopn}
\usepackage{algorithm}
\usepackage{algpseudocode}
\usepackage{authblk}
\usepackage{enumerate}%
\usepackage{enumitem}

\usepackage[font={it,scriptsize}, labelfont=bf, hangindent=1em]{caption}

\usepackage{url}

\usepackage[capitalize]{cleveref}

\usepackage{pgffor}
\foreach \a in {q,w,e,r,t,y,u,i,o,p,a,s,d,f,g,h,j,k,l,z,x,c,v,b,n,m,Q,W,E,R,T,Y,U,I,O,P,A,S,D,F,G,H,J,K,L,Z,X,C,V,B,N,M}{ 
	\expandafter\xdef\csname v\a\endcsname {{\noexpand\mathbf{\a}}}
}

\newcommand{\vectorizeGR}[1]{\bm{#1}}
\newcommand{\vOmega}{\vectorizeGR{\Omega}}
\newcommand{\vlambda}{\vectorizeGR{\lambda}}

\DeclareMathOperator{\amin}{argmin}
\DeclareMathOperator{\amax}{argmax}
\DeclareMathOperator{\prox}{prox}
\DeclareMathOperator{\proj}{proj}

\def\lp{\left(}
\def\rp{\right)}
\def\lq{\left[}
\def\rq{\right]}

\def\la{\left\langle}
\def\ra{\right\rangle}
\newcommand{\Id}{\vI_{\rm {\bm d}}}

\newcommand{\argmin}[1]{\ensuremath{\underset{\substack{{#1}}}{\amin}}\,\,}
\newcommand{\argmax}[1]{\ensuremath{\underset{\substack{{#1}}}{\amax}}\,\,}

\newcommand{\kiter}[1]{#1^{k}}
\newcommand{\kkiter}[1]{#1^{k+1}}
\newcommand{\vwk}{\kiter{\vw}}
\newcommand{\vwkk}{\kkiter{\vw}}
\newcommand{\vxk}{\kiter{\vx}}
\newcommand{\vxkk}{\kkiter{\vx}}

\newcommand{\vlambdak}{\kiter{\vlambda}}
\newcommand{\vlambdakk}{\kkiter{\vlambda}}
\newcommand{\den}{D_{\bm{\vartheta}}}
\newcommand{\Qt}{Q_{\bm{\vartheta}}}

\newcommand{\rd}{\mathbb{R}}
\newcommand{\ds}{\displaystyle}

\newtheorem{prp}{Proposition}
\newtheorem{remark}{Remark}

\title{Plug and Play Splitting Techniques for Poisson Image Restoration}

\author[1,2]{Alessandro Benfenati}
\affil[1]{\small Environmental and Science Policy Department, University of Milan, Via Celoria 2, 20133 - Milan, Italy}
\affil[2]{Gruppo Nazionale Calcolo Scientifico, INDAM, Piazzale Aldo Moro, 5 00185 – Rome, Italy}
\date{}

\newcommand{\pnp}{PnPSplit$^{+}$}

\begin{document}
\maketitle

\begin{abstract}
	{Plug and Play (PnP) methods achieve remarkable results in the framework of image restoration problems for Gaussian data. Nonetheless, the theory available for the Gaussian case cannot be extended to the Poisson case, due to the non-Lipschitz gradient of the fidelity function, the Kullback-Leibler functional, or the absence of closed-form solution for the proximal operator of such term, leading to employ iterative solvers for the inner subproblem. In this work we extend the idea of PIDSPLIT+ algorithm, exploiting the Alternating Direction Method of Multipliers, to PnP scheme: this allows to provide a closed form solution for the deblurring step, with no need for iterative solvers. The convergence of the method is assured by employing a firmly non expansive denoiser. The proposed method, namely \pnp, is tested on different Poisson image restoration problems, showing remarkable performance even in presence of high noise level and severe blurring conditions.}
\end{abstract}

\section{Introduction}

Imaging problems {arise} in several scientific fields, such as Medicine \cite{a16060270,hansen2021computed}, Astronomy \cite{beltramo2020joint,dabbech2022first,benfenati2016deconvolution}, and Microscopy \cite{calisesi2022compressed,zunino2023reconstructing,10.1093/bioinformatics/btab808}. The mathematical model underlying the physics process is shared among all these disciplines \cite{bertero2021introduction}, and it reads as
\begin{equation}
	\label{eq:linMod}
	\vg = \mathcal{N}\lp\vH\,\vx^\star + b\rp, 
\end{equation}

where $\vx^\star\in\rd^n$ {denotes} the ground truth image, $\vH\in\rd^{m\times n}$ is a linear operator perturbing the data, $b\in\rd^+$ is a known background parameter,  $\vg\in\rd^m$ the recorded image and $\mathcal{N}$ denotes the statistical noise on recorded data. {Classical examples of noise model include Gaussian noise \cite{bertero2021introduction}, Salt\&Pepper noise \cite{zhong2021image}, Speckle noise \cite{cheng2024diffusion}, Poisson noise \cite{10.1088/2053-2563/aae109} and mixture of Poisson and Gaussian noise \cite{lanza2014image}.} The operator $\vH$ is also called Point Spread Function (PSF), since its representation is the registered image of a point source; classical hypotheses on $\vH$, abiding by real life systems properties, are that $\vH^\top\bm{1}=\bm{1}$ and $\sum_{ij}H_{ij}=1$. The aim of image restoration problems is to recover an estimation of $\vx^\star$ given the registered data $\vg$ and the operator $\vH$. When the recorded data $\vg$ is affected by Poisson noise, under a Bayesian framework, that is adopting a \emph{maximum a posteriori} approach \cite{10.1088/2053-2563/aae109, vardi1985statistical, geman1984stochastic}, one is led to solve the optimization problem
\begin{equation}
	\label{eq:varprob}
	\argmin{\vx\in\rd^n_+} KL(\vH\,\vx+b, \vg) + \beta\,R(\vx),
\end{equation}
where $KL$ is the generalized Kullback-Leibler functional
$$
{KL(\vH\,\vx+b, \vg) = \sum_i g_i\log\lp\frac{g_i}{\lp\vH\vx\rp_i+b}\rp + \lp\vH\vx\rp_i +b -g_i.}
$$
The operations are intended component wise, and one assumes $0\log(0)=0$. In particular, $KL(\cdot, \vg)$ is a proper, convex and differentiable functional. The function $R$ is the regularization functional and its role is to preserve the desired characteristics on the estimated solution, such as sharp edges or sparseness, and to control the influence of the noise on the estimated solution. Common choices for $R$ {includes} proper, lower semi-continuous (l.s.c.) convex function, such as $\ell_2$ regularization, which goes also under the name Tikhonov regularization \cite{tikhonov1963solution, golub11998tikhonov} or Ridge Regression in other  frameworks \cite{10.1371/journal.pone.0245376}, $\ell_1$ norm for promoting sparsity on the solution \cite{gapyak2025ell}, a convex combination of $\ell_2$ and $\ell_1$ norms, commonly referenced to as Elastic Net \cite{benfenati2020regularization}. Another popular choice is the Total Variation functional \cite{caselles2015total}, for promoting sharp edges, and its offsprings \cite{morotti2024space,landi2025scaled}. The parameter $\beta\in\rd^+$ is the regularization parameter and balances the trade--off between the $KL$ and $R$. A common requirement in imaging problems is that the solution components are non-negative, since they represents pixels' intensity: therefore, the estimated solution is required to belong to the non negative orthant $\rd^n_+$.

The literature presents a plethora of variational methods to solve this particular instance of restoration problem: among them one can find gradient approaches \cite{chambolle2016introduction} and the related variation \cite{crisci2023hybrid,Jianjun2018,savanier2022unmatched}, Bregman iterative methods \cite{li2023poisson,benfenati2013inexact,rossignol2022bregman}, proximal approaches \cite{chouzenoux2024variational}. The Alternating Direction Method of Multipliers (ADMM) has gained a predominant role in image restoration problem \cite{figueiredo2010restoration, boyd2011distributed,di2021directional}, showing particularly interesting results in managing optimization problems with linear constraints. 

The seminal work \cite{6737048} introduced a novel approach, called  Plug and Play (PnP) technique. This strategy consists of solving optimization problems, whose objective functional encompasses two terms. Employing splitting techniques as ADMM, the authors in \cite{6737048} observed that the update for one of the variables reads actually as a Gaussian denoising step: therefore, they propose to substitute such updating step with an off-the--shelf denoiser $D$, such as Block-Matching and 3D Filtering (BM3D) \cite{dabov2007image}, Nonlocal Mean Filter (NLM) \cite{buades2011non}. Modern approaches encompass also the usage of deep neural networks, tailored for Gaussian denoising \cite{7839189}. 

The main hypothesis is that such denoiser is the proximal operator of some function $R$: the numerical experience showed the remarkable results of this approach. The research interest then moved to {investigate the} theoretical hypothesis to have on the denoiser for assuring the convergence of PnP: indeed, fixed point theory tells us that such denoiser needs to be firmly non expansive \cite{ryu2019plug,9054731}, but unfortunately most of the employed denoisers do not {fulfil} this requirement \cite{combettes2020deep}, despite their impressive performance results. Even classical neural networks, that show remarkable performances in Gaussian denoising tasks, cannot satisfy this requirement, unless properly trained with tailored loss function \cite{doi:10.1137/20M1387961}. The scientific research explored the control of the Lipschitz constant of the neural network \cite{ducotterd2024improving, hertrich2021convolutional,yoshida2017spectralnormregularizationimproving}, but the quality of such control is not strong enough to ensure the convergence property and moreover the computational cost is rather high. In \cite{jimaging10020050} the PnP framework has been addressed by considering it as a constrained problem under an ADMM approach, where a discrepancy principle is used in reformulating the problem. This approach allows to automatically chose the regularization parameter. Different techniques have been explored to assure convergence of PnP method: bounded denoisers assure fixed point convergence \cite{7744574}; in \cite{9473005} an incremental version of PnP with explicit requirements on the denoiser, namely its firmly non expansiveness, assures the convergence while maintaining scalability in terms of speed and memory. In \cite{9380942} under the hypotheses of the denoiser being averaged and the convexity of the data fidelity term the PnP scheme converges, and moreover it is shown that some of the employed denoisers are indeed the proximal operator of particular functions, \emph{e.g.}, the NLM is the prox of a quadratic convex function. 

One has to mention alternative approaches to PnP, which try to address the theoretical issues posed by PnP. The Regularizaton by Denoising (RED) method \cite{romano2017little} is among them, it tries to overcome the PnP limitations by requiring the denoiser to have a symmetric Jacobian and to be locally homogeneous: unfortunately, although the theoretical framework is very rich and interesting, the majority of the employed denoiser do not satisfy this requirements. RED has been then investigated from different points of view: it has been reformulated  \cite{cohen2021regularization} as a  constrained optimization problem (RED-PRO), where  the least square minimum is projected on the fixed-point sets of demicontractive denoisers, which reveal to be convex sets. In \cite{10415495} the RED-PRO has been reversed following a discrepancy principle, leading to a constrained RED approach (CRED): the RED functional is minimized under the discrepancy between the recovered solution and the data $\vg$. Deep equilibrium models have been recently studied for addressing Poisson image restoration \cite{daniele2025deepequilibriummodelspoisson}.

A further step was done considering Gradient Step Denoisers \cite{hurault2022gradient}, where the denoising step is carried out by subtracting to the current image the gradient of a parametrized function $g_\vartheta$ : a classical and performant choice is $g_\vartheta(\vx) = 1/2\|\vx-n(\vx)\|_2^2$, where $n$ is a denoising neural network. This particular strategy allows for a more solid theoretical convergence property and, from the practical point of view, it is possible to learn the denoiser without compromising the numerical performance.

Most of the previous research on PnP methods focused on data corrupted by Gaussian noise. Image corrupted by Poisson noise presents different challenges, mainly for the presence of the Kullback Liebler divergence as part of the objective functional. The seminal work \cite{hurault2023convergent} adopt a Bregman approach for designing a tailored method for deblurring and denoising tasks in presence of Poisson noise: the remarkable numerical results are supported by solid theoretical result. Adopting a different strategy, in \cite{9616253} a novel denoisier is created for Poisson data employing a denoiser based on Schroedinger equation's solution from quantum physics. An ADMM approach is adopted in \cite{ROND201696}, showing reliable results also in presence of high level Poisson noise. Beside variational methods, the authors in \cite{klatzer2025efficientbayesiancomputationusing} explore Bayesian approaches, in particular Langevin approaches, for addressing image restoration for Poisson data.

The variational methods previously mentioned show remarkable results in term of reconstruction, both in denoising and deblurring tasks, and rely on solid theoretical basis. Nonetheless, all of them rely on iterative methods for solving the deblurring step, meaning that either one has to accept an inexact solution to the inner problem or wait for the convergence of the inner iterative procedure. In this work, instead,  the split Bregman approach presented in \cite{Setzer2012} is exploited, \emph{i.e.}, coupling it with the PnP idea of substituting the proximity operator with an off-the--shelf denoiser, chosen to satisfy the firmly non expansive property. The split Bregman technique allows to avoid the usage of iterative methods for the deblurring step, by solving a trivial Least Square minimization problem, which possesses the nice property of having an unique solution. This significantly reduces the computational cost and, indirectly, the computational time. When the chosen denoiser satisfy the firmly non expansiveness hypothesis, one can extend the theoretical result of \cite{Setzer2012} for proving the convergence of the proposed scheme. The proposed method is then compared with state of the art algorithms and tested under different blurring conditions and Poisson noise levels.

This work is organized as follows. \cref{sec:PnPMethods} initially provides a background on ADMM and on Plug and Play methods, providing convergence results for the former and setting the notations used throughout the work. \cref{sec:pnpsplitp} presents the proposed method, providing the convergence result. \cref{sec:numexp} assesses the performance of the proposed method, comparing it with state of the art algorithms, testing under extreme {perturbation} conditions {and under different blurring operator} and, finally, employing a denoiser which does not satisfy the theoretical requirements for convergence. Eventually, \cref{sec:conclusions} draws the final considerations and consider possible future extensions of this work.

\paragraph{Notation.} The set $\rd^n$ denotes the real vector space of dimension $n$, $\rd^{m\times n}$ denotes real matrices with $m$ rows and $n$ columns. Bold capital symbols ($\vA,\vOmega,\dots$) denotes matrices, bold small symbols ($\vx,\vlambda,\dots$) denotes vectors. {For a vector $\vx\in\rd^n$, $\vx\geq\bm{0}$ means that each element of $\vx$ is greater or equal to zero. {The set $\rd^n_+$ is the non negative orthant of $\rd^n$: $\rd^n_+ = \{\vx\in\rd^n| \vx\geq 0\}$}.} Italic and Greek letters denote scalars in $\rd$. $\|\cdot\|_p$ stands for the $\ell_p$ norm. $\proj_A$ denotes the projection onto the set $A$. The set $\Gamma_0$ denotes the set of convex, proper and lower semi continuous (l.s.c.) functions. {The indicator function of a set $C$ is denoted with $\iota_C(\vx)$, where}
$$
\iota_C(\vx) = \begin{cases}
	0 & \text{if}\, \vx\in C\\
	+\infty & \text{otherwise}
\end{cases}
$$The proximity operator {(called also proximal operator or just prox)} of a function $f$ at a point $\vc$ is denoted with $\prox_f(\vc)$, and it consists of
$$
\prox_f(\vc) = \argmin{\vx}f(\vx)+\frac12\|\vx-\vc\|_2^2.
$$

\section{Plug and Play Methods}
\label{sec:PnPMethods}
{Splitting methods address a general minimization problem of the form}
\begin{equation}
	\label{eq:varProb}
	\argmin{\vx} \psi(\vx) + \beta\varphi(\vM \vx)
\end{equation}
where $\varphi,\psi\in\Gamma_0$, with $\varphi$ being differentiable, and $\vM$ is a linear operator. Note that Problem \eqref{eq:varprob} can be cast in this form by setting $\psi\equiv \beta\,R,\, \varphi\equiv \beta^{-1}\,KL$ and $\vM=\vH$. 
{Introducing} $\vM\vx=\vw$, the problem can be recast as
$$
\argmin{\vx,\vw} \psi(\vx) + \beta\varphi(\vw), \quad \text{such that } \vM\vx=\vw.
$$ 
The new constraint can be embedded in the objective functional, leading to the Augmented Lagrangian:
$$
\mathcal{L}(\vx,\vw,\vlambda) = \psi(\vx) + \beta\varphi(\vw) + \frac{1}{2\gamma}\|\vM\vx-\vw+\vlambda\|_2^2-\frac{1}{2\gamma}\|\vlambda\|_2^2,
$$
where the substitution $\vlambda\gets \gamma\vlambda$ has been {made}, with a {slight} abuse of notation. This leads to solve the saddle point problem
\begin{equation}
	\label{eq:saddlepoint}
	\argmin{\vx,\vw}\argmax{\vlambda }\mathcal{L}(\vx,\vw,\vlambda).
\end{equation}
The popular Alternating Direction Method of Multipliers (ADMM) \cite{boyd2011distributed}  depicted in \cref{al:ADMM} allows {solving} \eqref{eq:saddlepoint} under suitable hypotheses.
\begin{algorithm}[ht]
	\caption{ADMM}
	\label{al:ADMM}
	\begin{algorithmic}
		\State{Set $\vx^0, \vw^0$ and $\vlambda^0$ accordingly, select the parameter $\gamma>0$.}
		\For{$k=0,1,\dots$}
		\State{$\vxkk = \argmin{\vx\geq \bm{0}} \psi(\vx) +\ds\frac{1}{2\gamma}\|\vM\vx-\vwk +\vlambdak\|_2^2$}
		\State{}
		\State{$ \vwkk = \argmin{\vw} \beta\varphi(\vw) + \ds\frac{1}{2\gamma}\|\vM\vxkk-\vw +\vlambdak\|_2^2 $}
		\State{}
		\State{$ \vlambdakk =\vlambdak +\vM\vxkk-\vwkk $}
		\EndFor
	\end{algorithmic}
\end{algorithm}

\begin{remark}
	The update step of the variable $\vw$ consists of the proximity operator of the function $\varphi$ computed at $\vM\vxkk+\vlambdak$.
\end{remark}

The following result \cite[Proposition 2.2]{Setzer2012} provides the convergence results for the sequences $\{\vlambdak\}_k$ and $\{\vwk\}_k$, and assess the requirements to be met for having the iterates $\{\vxk\}_k$ solve the primal problem \eqref{eq:varProb}.

\begin{prp}[\cite{Setzer2012}]
	\label{prp:convergence} For any starting point and for any $\gamma\in\rd^+$ the sequences $\{\vlambdak\}_k$ and $\{\vwk\}_k$ generated by \cref{al:ADMM} converge. The sequence $\{\vxk\}_k$ calculated by \cref{al:ADMM} converges to a solution of the primal problem \eqref{eq:varProb} if one of the following conditions
	is met:
	\begin{enumerate}
		\item The primal problem has one and only one solution
		\item The optimization problem
		$$
		\argmin{\vx} \psi(\vx) + \frac{1}{2\gamma}\|\vM\vx-\hat\vw +\hat\vlambda \|_2^2
		$$
		has a unique solution, where 
		$$
		\hat\vw = \lim_{k\to\infty} \vwk, \quad \hat\vlambda = \lim_{k\to\infty} \vlambdak
		$$
	\end{enumerate}
\end{prp}

The seminal work \cite{6737048} observed that the update rule for $\vw$ in \cref{al:ADMM} can be interpreted as a Gaussian denoising step on the variable $\vw$, with a regularization function $\varphi$. Therefore, they proposed to plug in an off-the-shelf Gaussian denoiser $\mathcal{D}_{\gamma\beta}$ instead of the proximal step, where $\gamma\beta$ is the standard deviation of the Gaussian noise to be removed. The method takes the name of Plug and Play (PnP) and it is depicted, in its general formulation, in \cref{al:Pnp}. Some examples {for the denoiser used in PnP schemes} are BM3D \cite{dabov2007image} or Nonlocal Mean Filter \cite{buades2011non}, or trained deep neural networks \cite{doi:10.1137/20M1387961}. The advantage of this strategy is twofold: one does not need to select a priori  a regularization function $\varphi$ and furthermore, once chosen, one can avoid to compute the proximal operator of $\varphi$, via a direct formula-as in the $\ell_1$ case-or via an iterative method, \emph{e.g.}, when $\varphi$ is the Total Variation regularization.
\begin{algorithm}[ht]
	\caption{Plug and Play}
	\label{al:Pnp}
	\begin{algorithmic}
		\State{Set $\vx^0, \vw^0$ and $\vlambda^0$ accordingly, select the parameter $\gamma>0$.}
		\For{$k=0,1,\dots$}
		\State{$\vxkk = \argmin{\vx\geq \bm{0}} \psi(\vx) +\ds\frac{1}{2\gamma}\|\vM\vx-\vwk +\vlambdak\|_2^2$}
		\State{}
		\State{$ \vwkk = \mathcal{D}_{\gamma\beta}(\vM\vxkk+\vlambdak)$}
		\State{}
		\State{$ \vlambdakk =\vlambdak +\vM\vxkk-\vwkk $}
		\EndFor
	\end{algorithmic}
\end{algorithm}
This strategy proved to achieve remarkable results in terms of reconstruction quality and computational time: the numerical experience \cite{6737048,cascarano2022plug,pendu2023preconditioned} showed that this method is able to exploit both the properties of the original variational model and the noise-removal abilities of the chosen denoiser. 

{Nonetheless, such an approach does not come without presenting several challenges. Consider the case in which the chosen denoiser is the prox operator of an \textit{unknown} function: this amounts to \emph{implicitly} defining the primal problem to solve and the relative objective function, the latter consisting of the data fidelity $\psi$ (selected according to the noise affecting the data) and of a regularization function $\rho$ such that $\mathcal{D}(\cdot)=\prox_\rho(\cdot)$. The implicit optimization problem is therefore
	$$
	\argmin{\vx} \psi(\vx) + \rho(\vM\vx).
	$$	
	For example, if the denoiser is the soft thresholding operator, then it is a common knowledge that the function $\rho$ is the $\ell_1$ norm; it can be proven that a class of linear denoisers, such as kernel and symmetric denoisers, can be expressed as proximity operators of some convex functions \cite{9418516}}. 

{On the other hand, if the chosen denoiser $\mathcal{D}$ does not meet the requirements to be the proximal operator of a function, then the convergence of PnP method is no longer related to the primal problem, the focus is on consensus equilibrium formulation: this amounts to have 
	$$
	\tilde\vx = \prox_{\psi}(\tilde\vx-\tilde\vlambda),\quad \text{and} \quad \tilde\vx = \mathcal{D}(\tilde\vx+\tilde\vlambda),
	$$
	where $\tilde\vx$ is the restored solution and $\tilde\vlambda$ can be interpreted as the noise component removed from the data. For a more thoughtful discussion, see \cite{10004791}.}

{This work focuses on the latter case: select and/or train a denoiser $\mathcal{D}$ which is a firmly non expansive operator, and it is therefore the resolvent of a maximally monotone operator \cite{9054731,ryu2019plug}, aiming to exploit classical convergence results on ADMM.}

\section{Proposed Method}
\label{sec:pnpsplitp}
{In the Poisson case, PnP methods typically require an iterative solver for the inner deblurring subproblem, see for example \cite{hurault2023convergent,9616253,ROND201696}. This work exploits a clever strategy originally introduced in \cite{figueiredo2010restoration,Setzer2012}, which allows one to avoid such iterative procedures. Such strategy is then combined with the PnP approach, and, by  carefully choosing the denoiser, the convergence behaviour of the proposed method is ensured.}

\subsection{Previous Work: PIDSPLIT+}
The authors in \cite{Setzer2012} generalized the method proposed in \cite{figueiredo2010restoration}, using a common but clever mathematical trick: adding 0 to the objective functional, which in this case amounts to the scalar product of $\vx$ and the zero vector. The optimization problem \eqref{eq:varprob} is slightly modified by adding the term $\la\vx,\bm{0}\ra$ and by {introducing} the indicator function ${\iota_{\rd^n_+}}$:
\begin{equation}
	\label{eq:VarProbSetzer}
	\argmin{\vx} \la\vx,\bm{0}\ra + KL(\vH\,\vx+b, \vg) + \beta\,R(\vx) + {\iota_{\rd^n_+}(\vx)}.
\end{equation}
Introducing the matrix $ \vM= \lp\vH^\top, \Id, \Id\rp^\top$ the problem \eqref{eq:VarProbSetzer} can be restated as
\begin{eqnarray*}
	\label{eq:augLagr}
	\argmin{\vx,\vw} \la\vx,\bm{0}\ra + \varphi(\vw),\quad \text{s.t.}\quad \vM\,\vx = \vw\Leftrightarrow\begin{pmatrix}
		\vH\\\Id\\\Id
	\end{pmatrix}\vx = \begin{pmatrix}
		\vw_1-b\\\vw_2\\\vw_3
	\end{pmatrix}
\end{eqnarray*}
which abides to the formulation in \eqref{eq:varProb} with
$$
\varphi(\vw) = KL(\vw_1,\vg) + \beta\,R(\vw_2) + {\iota_{\rd^n_+}}(\vw_3), \quad \psi(\vx) = \la\vx,\bm{0}\ra.
$$
This can be easily generalized when the {regularization} function $R$ encompasses a linear operator $\vL$, as $R(\vL\vx)$: the matrix $\vM$ reads hence as $\vM^\top=\lp\vH^\top, \vL^\top, \Id\rp^\top$. The natural next step is to apply \cref{al:ADMM} to this problem. In particular, the update step for $\vxkk$ reads as
$$
\vxkk = \argmin{\vx}\la\vx,\bm{0}\ra + \frac{1}{2\gamma}\|\vM\vx-\vwk+\vlambdak\|_2^2
$$
which amounts to {solving}
$$
\lp\vH^\top\vH + 2\Id\rp\vx= \vH^\top\lp\vwk-\vlambdak\rp
$$
The system matrix is square and non singular: therefore, it has one and only one solution: this leads to satisfy condition ii) of \cref{prp:convergence}, and therefore the whole method converges. Moreover, assuming the usual hypotheses on the PSF $\vH$, the solution of such system can be easily computed by means of FFT.

Due to the separability of the components of the vector $\vw$, the update for $\vwkk$ is straightforward:
\begin{itemize}
	\item The component $\vwk_1$ is computed as the proximal operator of the Kullback--Leibler functional $KL(\cdot,\vg)$:
	\begin{eqnarray*}
		\vwkk_1 &=& \argmin{\vw_1} KL(\vw_1,\vg) + \frac{1}{2\gamma}\|{\vH}\,\vxkk  -\vw_1 +\vlambdak_1\|_2^2\\
		&=& \prox_{\gamma\,KL(\cdot+b,\vg)}\lp\vH\vxkk+\vlambdak_1\rp\\
		&=& \frac12\lp\ \vH\vxkk+b+\vlambdak_1-\gamma +\sqrt{\lp\vH\vxkk+b+\vlambdak_1-\gamma\rp^2+4\gamma\,\vg}\rp,
	\end{eqnarray*}
	where the operations are component-wise.
	
	\item The component $\vw_2$ {is given by} the proximity operator of the regularization function:
	\begin{eqnarray*}
		\vwkk_2 &=& \argmin{\vw_2} \beta\,R(\vw_2) + \frac{1}{2\gamma}\|\vxkk-\vw_2+\vlambdak_2\|_2^2\\
		&=& \prox_{\beta\gamma R} \lp\vxkk+\vlambdak_2\rp
	\end{eqnarray*}
	
	\item The third element of $\vw$ is the projection on the non-negative orthant:
	\begin{eqnarray*}
		\vwkk_3 &=& \argmin{\vw_3} {\iota_{\rd^n_+}(\vw_3)} + \frac{1}{2\gamma}\|\vxkk-\vw_3+\vlambdak_3\|_2^2\\
		&=& \proj_{{\rd^n_+}} \lp\vxkk+\vlambdak_3\rp
	\end{eqnarray*}
	
\end{itemize}
These steps are gathered in \cref{al:pidsplit+}, together with the final update{s} of the Lagrangian multipliers ({which are not listed} one by one for the sake of brevity). 
\begin{algorithm}[ht]
	\caption{PIDSPLIT+\cite{Setzer2012}}
	\label{al:pidsplit+}
	\begin{algorithmic}
		\State{Set $\vx^0, \vw^0$ and accordingly $\vlambda^0$; select the parameter $\gamma>0$.}
		\For{$k=0,1,\dots$}
		\State{$\vxkk = \lp\vH^\top\vH + 2\Id\rp^{-1}\lq\vH^\top\lp\vwk_1-\vlambdak_1\rp + \vwk_2-\vlambdak_2 + \vwk_3-\vlambda_3^k\rq$}
		\State{}
		\State{$ \vwkk_1 =\ds\frac12\lp\ \vH\vxkk+b+\vlambdak_1-\gamma +\sqrt{\lp\vH\vxkk+b+\vlambdak_1-\gamma\rp^2+4\gamma\,\vg}\rp$}
		\State{}
		\State{$\vwkk_2 = \prox_{\gamma R} \lp\vxkk+\vlambdak_2\rp$}
		\State{}
		\State{$\vwkk_3 = \proj_{{\rd^n_+}} \lp\vxkk+\vlambdak_3\rp$}
		\State{}
		\State{$ \vlambdakk =\vlambdak +\vM\vxkk-\vwkk $}
		\EndFor
	\end{algorithmic}
\end{algorithm}

\subsection{\pnp}
{With the aim of adopting a PnP approach}, the update rule for $\vw_2$ {takes again the form} of a Gaussian denoising step: therefore, following the original PnP idea, one employs a Gaussian denoiser $\mathcal{D}_{\beta\gamma}$ in place of the proximal operator of $R$. This choice leads to a novel version of this splitting algorithm, {called} \pnp, which exploits the splitting idea of \cite{Setzer2012} and the possibility to select an off-the-shelf denoiser, instead of meticulously selecting a regularization function $R$ and devising tailored algorithm for computing its proximity operator.
\begin{algorithm}[htbp]
	\caption{\pnp}
	\label{al:pnpsplit+}
	\begin{algorithmic}
		\State{Set $\vx^0, \vw^0$ and accordingly $\vlambda^0$; select the parameter $\gamma>0$.}
		\For{$k=0,1,\dots$}
		\State{$\vxkk = \lp\vH^\top\vH + 2\Id\rp^{-1}\lq\vH^\top\lp\vwk_1-\vlambdak_1\rp + \vwk_2-\vlambdak_2 + \vwk_3-\vlambda_3^k\rq$}
		\State{}
		\State{$ \vwkk_1 =\ds\frac12\lp\ \vH\vxkk+b+\vlambdak_1-\gamma +\sqrt{\lp\vH\vxkk+b+\vlambdak_1-\gamma\rp^2+4\gamma\,\vg}\rp,$}
		\State{}
		\State{$\vwkk_2 = \mathcal{D}_{\beta\gamma}\lp\vxkk+\vlambdak_2\rp$}
		\State{}
		\State{$\vwkk_3 = \proj_{{\rd^n_+}} \lp\vxkk+\vlambdak_3\rp$}
		\State{}
		\State{$ \vlambdakk =\vlambdak +\vM\vxkk-\vwkk $}
		\EndFor
	\end{algorithmic}
\end{algorithm}
The main advantage of this approach is that the deblurring step is computed with a direct explicit formula, without relying on an iterative solver, reducing significantly the computational cost and time.

The denoiser, however, should be properly trained (or selected) in order to {ensure} the convergence behavior of \pnp algorithm: this requires that such denoiser is firmly non expansive \cite{9054731}, as already {discussed} in \cref{sec:PnPMethods}. If the selected denoiser is a  convolutional neural network, the latter network can be trained in order to satisfy this requirement, as presented in \cite{doi:10.1137/20M1387961}. The strategy to train such a denoiser is briefly recalled below.

Consider the differential operator  $\Qt = 2\den-\Id$, where $\bm{\vartheta}$ are the trainable parameters: classical results state that the denoiser $\den$ is firmly non expansive if and only if $\Qt$ is non expansive: therefore the training of $\den$ should be carried out by solving
$$
\argmin{\bm\vartheta} \sum_i L(\den(\vx_i),\vy_i) \quad \text{such that } \Qt \text{ is non expansive},
$$ 
where $\{\vx_i,\vy_i\}_i$ is the dataset of noisy and clean images for the training and $L$ is the loss function, usually  MSE score, used for training. The authors in \cite{doi:10.1137/20M1387961} assume that $\Qt$ is differentiable for any $\bm{\vartheta}$, therefore the requirement for the non expansiveness amounts to 
$$
\|\nabla \Qt(\vx)\|\leq 1 \quad\forall \, \vx.
$$
Unfortunately, this cannot be met for each $\vx$, hence in \cite{doi:10.1137/20M1387961} this constraint is imposed on every line $[\vx_i, \den(\vx_i)]$, \emph{i.e.}, on each point of the form $\tilde\vx_i= \delta_i\vx_i + (1-\delta_i)\den(\vx_i)$, with $\delta_i$ randomly drawn from an Uniform distribution on the interval $[0,1]$. The training phase for the denoiser reads hence as
\begin{equation}
	\argmin{\bm\vartheta} \sum_i L(\den(\vx_i),\vy_i) + \lambda \max\{\|\nabla \Qt(\tilde\vx_i)\|^2, 1-\varepsilon\},
	\label{eq:trainingPesquet}
\end{equation}
where $\lambda$ is a nonnegative regularization parameter and $\varepsilon\in(0,1)$ allows to control the constraints. The requirement on $\den$ to be differentiable can be overcome: automatic differentiation, the standard technique used in neural network training, allows one to consider denoisers implementing nonsmooth activation function such as ReLU (see \cite[Remark 3.3]{doi:10.1137/20M1387961} and \cite{bolte2021conservative,NEURIPS2021_70afbf22,NEURIPS2022_a9077da4,doi:10.1137/22M1541630} for more theoretical insights).

\begin{remark}
	\label{rem:spectralnorm}
	{The training in \eqref{eq:trainingPesquet} relies on computing the spectral norm $\|\nabla \Qt(\vx)\|$ for an image $\vx$: as explained in \cite[Remark 3.4]{doi:10.1137/20M1387961} this is accomplished by using the power iterative method and backpropagation for computing the Jacobian.}
\end{remark}

The convergence result for \cref{al:pnpsplit+} follows from \cref{prp:convergence}, considering the further requirement on the denoiser.

\begin{prp}{
		\label{prp:convpnp} Let $\den$ a firmly non expansive Gaussian denoiser, that is the resolvent of a Maximally Monotone Operator $A$. Set $\mathcal{D}_{\gamma\beta}=\den$ in \cref{al:pnpsplit+}. For any $\vx^0, \vw^0$ and for any $\gamma\in\rd^+$ the sequences $\{\vlambdak\}_k$ and $\{\vwk\}_k$ generated by \pnp converge. The sequence $\{\vxk\}_k$ computed in \pnp converges to $\tilde\vx$ such that
		$$
		0\in\vH^\top\nabla KL(\vH\tilde\vx+b;\vg) + A(\tilde\vx) + N_{\rd_+^n}(\tilde \vx),
		$$ 
		where $N_{\rd_+^n}$ is the normal cone to $\rd^n_+$, if one of the following conditions is met:
		\begin{enumerate}[label=\roman*)]
			\item The primal problem has one and only one solution
			\item The optimization problem
			\begin{equation}\label{eq:mineProp}
				\argmin{\vx} \psi(\vx) + \frac{1}{2\gamma}\|\vM\vx-\hat\vw +\hat\vlambda \|_2^2, \quad 
				\hat\vw = \lim_{k\to\infty} \vwk, \quad \hat\vlambda = \lim_{k\to\infty} \vlambdak 
			\end{equation}
		\end{enumerate}
		has an unique solution.}
\end{prp}

\begin{proof}{
		Since $\den$ is firmly non expansive, the update step in $\vw_2$ amounts to compute the resolvant operator of a maximally monotone operator $A$:
		$$
		\kkiter{\vw}_2 = (A+\Id)^{-1}(\kkiter{\vx} +\vlambdak_2).
		$$
		The updates for $\kkiter{\vw}_1$ and $\kkiter{\vw_3}$ are the proximal operators of the KL and the projection on the non negative orthant, respectively: therefore, the whole update for $\kkiter{\vw}$ allows recasting the convergence proof of \pnp into the classical one of the ADMM. Moreover, from the optimality condition one has
		\begin{eqnarray*}
			0 &=& \vM^\top \lp\vM\tilde\vx-\tilde\vw+\tilde\vlambda\rp\\
			0&\in&\nabla KL(\vw_1;\vg) + A(\vw_2) + N_{\rd_+^n}(\vw_3) - \frac1\gamma\lp\vM\tilde\vx-\tilde\vw+\vlambda\rp\\
			0&=&\vM\tilde\vx-\tilde\vw
		\end{eqnarray*}
		then, plugging in the constraint $0=\vM\tilde\vx-\tilde\vw$
		\begin{eqnarray*}
			0 &=& \vM^\top \tilde\vlambda\\
			0&\in&\nabla KL(\vH\tilde\vx+b;\vg) + A(\tilde\vx) + N_{\rd_+^n}(\tilde\vx )-\gamma^{-1}\tilde\vlambda,
		\end{eqnarray*}
		leading thus to 
		$$
		0\in\vH^\top\nabla KL(\vH\tilde\vx+b;\vg) + A(\tilde\vx) + N_{\rd_+^n}(\tilde\vx ).
		$$
		Moreover, as already stated for the PIDSplit+ algorithm, the update step for $\vxkk$ consists of solving a square linear system whose matrix is non singular, therefore the solution is unique: this amounts to satisfy item \textit{ii)}.}
\end{proof}

\begin{remark}{
		The proof of \cref{prp:convpnp} regards the case of firmly non expansive denoisers. If one manages to employ a $\den$ such that it is an actual proximal operator, \emph{i.e.}, $\den=\prox_\rho$ for some unknown $\rho$, then the convergence follows directly from \cref{prp:convpnp}. In this case, however, the limit point $\tilde\vx$ solves a primal problem whose objective functional is unknown, so that the objective functional of the  corresponding primal problem is  $KL(\vH\cdot;\vg) + \rho(\cdot) + \iota_{\rd_+^n}$.}
\end{remark}

\begin{remark}
	\label{rem:adaptiveGamma}{
		The performance of ADMM is related to the choice of the parameter $\gamma$: the literature \cite{He2000,Wang2001,Zanni2015} presents an adaptive strategy to overcome this issue. Such strategy relies on two quantities, namely the primal and dual residuals:
		\begin{equation}
			\label{eq:pdRes}
			\begin{split}
				\vp^k &= \vM\vxk-\vwk - b\\
				\vs^k &= \frac{1}{\gamma} \vM^t\lp\vwk-{\vw}^{k-1}\rp.
			\end{split}
		\end{equation}
		These two quantities provides insights on the upper bound on the absolute error among the objective function and its minimum value at the current iterate \cite{Zanni2015}. These residuals are employed to design an adaptive strategy for selecting the value for $\gamma$, and the convergence of ADMM is assured provided that $\gamma$ stabilizes after a fixed number of iteration. This strategy is investigated in \cref{ssec:gamma}.}
\end{remark}

\section{Numerical Experiments}
\label{sec:numexp}
This section is devoted to assess the performance of the proposed \pnp method. All the experiments have been carried out on a MacBook Pro equipped with M4 processors, in PyThorch environment. The code is available at \url{https://github.com/AleBenfe/PnPSplitPlus}.

The images employed for the experiments belong to the Set5 dataset \cite{bevilacqua2012low} {and BSD500 \cite{5557884}}. Each image is scaled to $[0,1]$, the Poisson noise has been {simulated} using a custom function, implemented via \verb"torch" library functions, which allows to select  the level $\nu$ of the noise affecting the image: the lower the value of $\nu$, the higher the level of the noise. The blurring operation is carried out via FFT.
\begin{figure}[htbp]
	\centering
	\includegraphics[width=.95\textwidth]{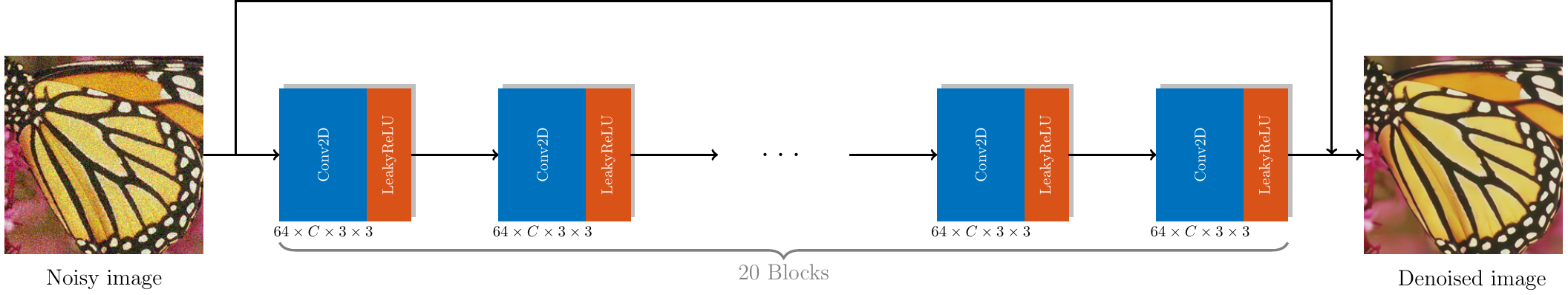}
	\caption{Sketch of the network used in the experiments. The network consists of 20 blocks of convolutional and LeakyReLU layers, with a skip connection between the input and the output. The number of channels is denoted by $C$: $C=3$ for RGB images, $C=1$ for black\&white images.}
	\label{fig:PesquetNet}
\end{figure}
The network employed as Gaussian denoiser in the update of the variable $\vw_2$ in \cref{al:pnpsplit+} is the {pretrained} deep convolutional network trained in \cite{doi:10.1137/20M1387961}{, available at \url{https://github.com/basp-group/PnP-MMO-imaging}, and depicted in \cref{fig:PesquetNet}. This network is inspired by the one in \cite{7839189}: it uses classical convolutional layers and the ReLU layers are replaced by LeakyReLU ones, while the Batch Normalization layers have been removed and a skip connection links the input with the output. This network has been trained by the original authors in two stages:}
\begin{itemize}
	\item {A pretraining is carried out on 50 000 images of the ImageNet dataset \cite{5206848}, using the Adam Algorithm,  and perturbing each minibatch of images in the following way:
		$$
		\vy_i=\vx_i + \sigma_i\varepsilon, \quad \varepsilon\sim\mathcal{N}(0,1), \, \sigma_i\sim \mathcal{U}[0,0.1].
		$$
		Therefore, the network is trained to perform \emph{blind denoising}. For this initial training, Jacobian regularization is not encompassed in the loss function \eqref{eq:trainingPesquet}, that is $\lambda=0$.}
	
	\item { Once this pretraining is accomplished, a refinement training step inserting Jacobian regularization in \eqref{eq:trainingPesquet}, considering  \cref{rem:spectralnorm}, is done on images of the BSD500 dataset \cite{937655}, perturbing them with Gaussian noise with a fixed standard deviation equal to 0.01. The resulting network is used in the experiments of this work.}
\end{itemize}

Four different measures are employed to assess the performances of the proposed method: the Mean Square Error (MSE), the relative error (RE) computed as $\|\vx^\star-\vx^{rec}\|/\|\vx^\star\|$, where $\vx^{rec}$ is the recovered image, the Peak Signal-to-Noise Ratio and the Structural Similarity Index (SSIM) \cite{Wang2004}. These indexes are computed on the last iterate $\vw_3^K$: at convergence, the iterates $\vw_2^K, \vw_3^K$ and $\vx^K$ should coincide, due to the constraints $\vM\vx=\vw$. {The initial iterate $\vx^0$ is set equal to $\vg$, and all the other variables accordingly: this setting is used among all the numerical experiments. The background $b$ is set to $10^{-3}$.}

\subsection{On the choice of $\beta\gamma$}
\label{ssec:betagamma}

{\Cref{al:pnpsplit+} requires using a denoiser that accounts for the standard deviation $\beta\gamma$ of the Gaussian noise on the current iterate. The parameter $\beta$ is actually encompassed by the unknown operator $A$ in \cref{prp:convpnp}, hence the parameter $\gamma$ plays the role of the standard deviation of the noise. In the numerical experiments presented here, the network employed as a denoiser has been trained on images affected by Gaussian noise with $\sigma=0.01$: therefore, the choice for the ADMM parameter is set as $\gamma=\sigma=0.01$.}

\subsection{Comparison with State-of-the-Art Algorithms}
\label{ssec:comparison}
This section is devoted to the comparison with state of the art algorithms. The first run of experiments is carried out for the comparison with the B-PnP algorithm, \cite{hurault2023convergent}, employing the code provided by the authors in the GitHub repository \url{https://github.com/samuro95/BregmanPnP}. Some slight modifications to the original code has been done, in order to run it on the same Apple machine and to have the same Poisson noise generator (\verb"torch.poisson" instead of \verb"numpy.random.poisson"). The comparison has been carried out on high level Poisson noise ($\nu=20$), and the images are blurred with a Gaussian PSF with $\sigma=1$. {Both algorithms are tested on two different settings, with the maximum number of iterations set to 2500 and 400}. B-PnP uses the PGD algorithm as inner solver.
\begin{figure}[htbp]
	\centering
	\newcommand\factor{0.24}
	\subfigure[$\vx^\star$, \emph{Butterfly} ]{\includegraphics[width=\factor\textwidth]{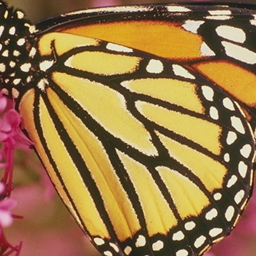}}\hfill\subfigure[ $\vg$, \emph{Butterfly} ]{\includegraphics[width=\factor\textwidth]{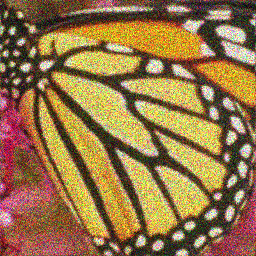}}\hfill\subfigure[B-PnP]{\includegraphics[width=\factor\textwidth]{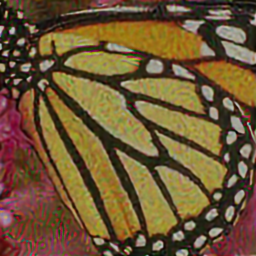}}\hfill\subfigure[\pnp]{\includegraphics[width=\factor\textwidth]{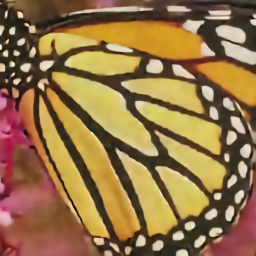}}
	\subfigure[$\vx^\star$, \emph{Tucano}]{\includegraphics[width=\factor\textwidth]{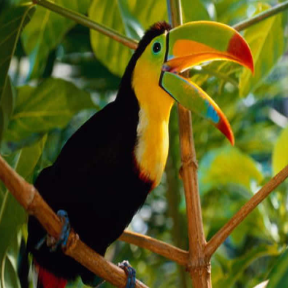}}\hfill\subfigure[$\vg$, \emph{Tucano}]{\includegraphics[width=\factor\textwidth]{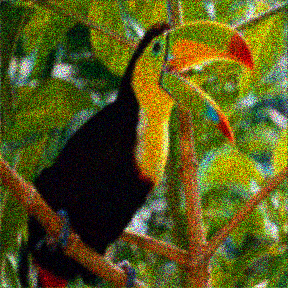}}\hfill\subfigure[B-PnP]{\includegraphics[width=\factor\textwidth]{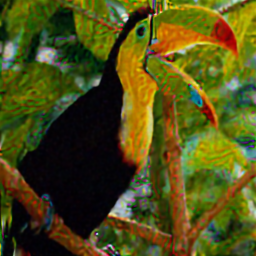}}\hfill\subfigure[\pnp]{\includegraphics[width=\factor\textwidth]{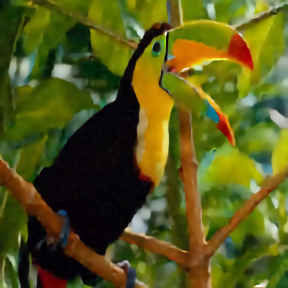}}
	\subfigure[ $\vx^\star$, \emph{Baby}]{\includegraphics[width=\factor\textwidth]{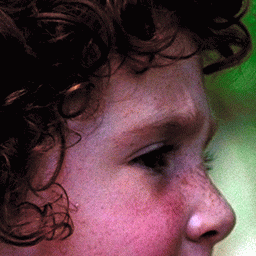}}\hfill\subfigure[ \emph{Baby}]{\includegraphics[width=\factor\textwidth]{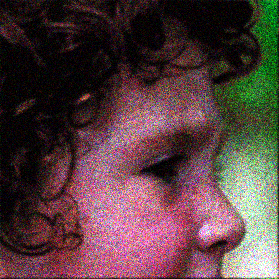}}\hfill\subfigure[B-PnP]{\includegraphics[width=\factor\textwidth]{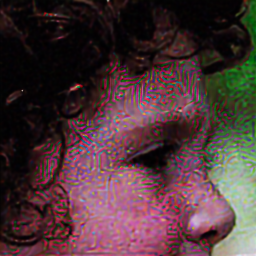}}\hfill\subfigure[\pnp]{\includegraphics[width=\factor\textwidth]{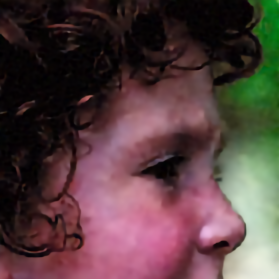}}
	\caption{Visual inspection of the recovered images provided by \pnp and B-PnP algorithms. First column: ground truth images. Second column: simulated recorded data, perturbed with a Gaussian PSF and Poisson noise at level 20. Third column: B-PnP reconstruction. Fourth column: \pnp reconstruction. Both algorithms have run for 2500 iterations. The B-PnP reconstructions suffer from the presence of some artifacts, while \pnp ones presents more smooth results. }
	\label{fig:comp}
\end{figure}

{\cref{tab:comparison} provides the performances indexes on the PSNR, MSE and SSIM. When the maximum number of iteration is reduced to 400 ({reflecting} the {default} setting for B-PnP) the methods provide similar results, in particular B-PnP improves both in terms of visual inspection and of indexes measure.  Nonetheless, \pnp achieves higher scores and shows a robust behaviour with respect to the number of iterations.} 
\cref{fig:comp} presents a visual inspection of the recovered images {obtained with 2500 iterations}: the {ones} provided by B-PnP method suffer from the presence of several artefacts, and in the case of the \textit{Butterfly} the image also from some kind of darkening effect. 
\begin{table}[htbp]
	\centering
	\begin{tabular}{l|c|c||c|c}
		& \multicolumn{2}{c||}{2500} & \multicolumn{2}{c}{400}\\
		\midrule
		& \pnp & B-PnP & \pnp & B-PnP \\
		\toprule
		& \multicolumn{4}{|c}{PSNR} \\
		\midrule
		
		\textit{Butterfly} 	& 25.18 & 22.14 & 25.07 & 23.43 \\
		\textit{Tucano} 	& 29.02 & 24.86 & 29.09 & 28.24 \\
		\textit{Baby} 		& 28.58 & 21.22 & 28.21 & 25.95 \\
		\midrule
		& \multicolumn{4}{|c}{MSE}\\ 
		\midrule
		\textit{Butterfly}	& 0.0030 & 0.0247 & 0.0032 & 0.0045 \\
		\textit{Tucano}		& 0.0012 & 0.0043 & 0.0013 & 0.0015 \\
		\textit{Baby} 		& 0.0014 & 0.0090 & 0.0015 & 0.0025 \\ 
		\midrule
		&\multicolumn{4}{|c}{SSIM} \\
		\midrule
		\textit{Butterfly} 	& 0.8518 & 0.6464 & 0.8372 & 0.7867 \\
		\textit{Tucano} 	& 0.8578 & 0.7090 & 0.8539 & 0.8246 \\
		\textit{Baby} 		& 0.6844 & 0.4887 & 0.6789 & 0.6362 \\
	\end{tabular}
	\caption{Comparison with B-PnP algorithm. Three different images have been considered, namely \textit{Butterfly}, \textit{Tucano} and \textit{Baby}. The proposed algorithm provides reliable performance measures; the B-PnP algorithm achieves better results when the maximum number of iterations is fixed to {400, as per default setting. The proposed method shows robustness with respect to the number of iterations.}}
	\label{tab:comparison}
\end{table}
{A further comparison with the B-PnP method is carried out on 20 images of the BSD500 dataset \cite{5557884}, the results are depicted in \cref{fig:BSDS500}. These plots show the behaviour of the PSNR among the 20 images: for each image 3 different runs have been considered, computing the average value and standard deviation for the PSNR, in order to avoid that a particular noise realization affects the performance evaluation. The shaded curved area represents the average PSNR $\pm$ the standard deviation. Two experiments were performed. In \cref{sfig:BSDS500_1000}, the maximum number of iterations was set to 1000 for both methods, whereas in \cref{sfig:BSDS500_1000}, it was set to 400.. \pnp is again more robust to the choice of maximum number of iterations, and B-PnP performs better when it is run for a low number of iterations, as presented in the original paper \cite{hurault2023convergent}. The computational time of \pnp method is lower than the one of B-PnP: the latter takes on average 7 seconds to run for 400 iterations, whilst B-PnP needs around 24 seconds. The difference increases for a larger number of iterations.} 
\begin{figure}[htbp]
	\begin{center}
		\subfigure[Max iters: 1000\label{sfig:BSDS500_1000}]{\includegraphics[width=0.45\textwidth]{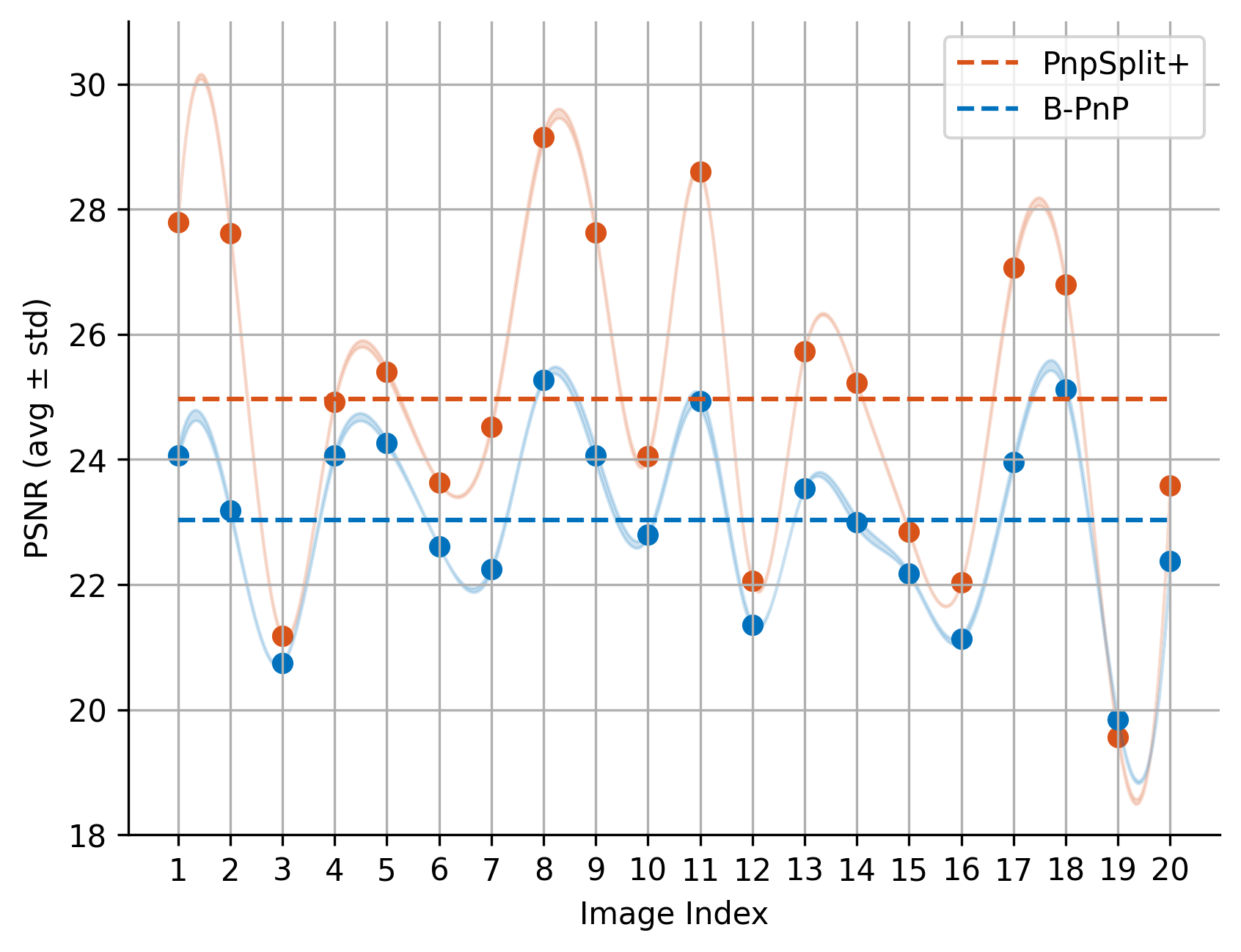}}\hfill\subfigure[Max Iters: 400\label{sfig:BSDS500_400}]{\includegraphics[width=0.45\textwidth]{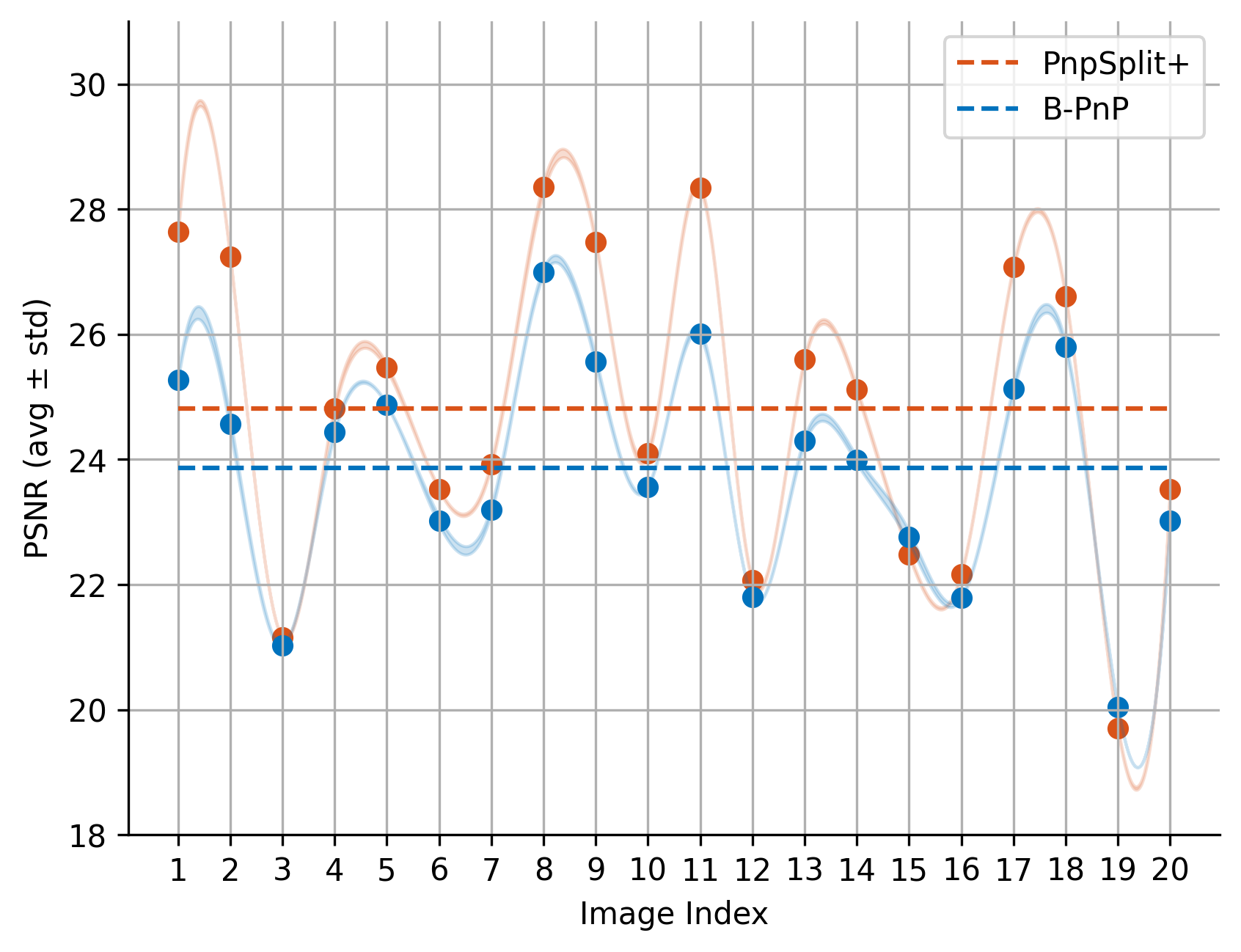}}
		\caption{PSNR assessment on 20 images of the BSDS500 dataset. The panels shows the results for \pnp and B-PnP when the maximum number of iteration is set to 1000 (left) and to 400 (right). The dots represent the average PSNR of the recovered images for each method. \pnp reveals to be quite robust with respect to the maximum number of iterations, and even when B-PnP runs with the optimal number of iterations \pnp is competitive. The curved behaviour is due to the  cubic spline used for plotting the results. }
		\label{fig:BSDS500}
	\end{center}
\end{figure}

The {third} run of experiments is done for comparing the \pnp with two other approaches: QAB-PnP \cite{9616253} and P$^4$IP \cite{ROND201696}. The test images employed in these experiments are modifications of the original ones, due to the memory constraints posed by the available MatLab code for QAB-PnP: the images are halved in both dimensions and transformed in gray scale images. The PSF inducing the blur is still a Gaussian one with $\sigma=1$ and the noise level is set to 20. The denoiser used in \cref{al:pnpsplit+} is {again the one} from \cite{doi:10.1137/20M1387961} with the appropriate number of input channels. Algorithm QAB-PnP run{s} on MatLab with no parallel implementation, the code is available at \url{https://github.com/SayantanDutta95/QAB-PnP-ADMM-Deconvolution}, while the Python code for P$^4$IP can be downloaded at \url{https://github.com/sanghviyashiitb/poisson-plug-and-play/tree/main}. 
\begin{table}[htbp]
	\centering
	\begin{tabular}{l|ccc|ccc|ccc}
		& \multicolumn{3}{c|}{\pnp} &\multicolumn{3}{c|}{QAB-PnP} &\multicolumn{3}{c}{P$^4$IP}\\
		\midrule
		&\textit{Butterfly} & \textit{Tucano}& \textit{Baby}&\textit{Butterfly} & \textit{Tucano}& \textit{Baby}&\textit{Butterfly} & \textit{Tucano}& \textit{Baby}\\ 
		\midrule
		MSE 	& 0.0057 & 0.0026 & 0.0012 & 0.1245 & 0.0621 & 0.0413 & 0.2213 & 0.1255 & 0.0862\\
		RE 		& 0.1526 & 0.1648 & 0.0937 & 0.2202 & 0.1902 & 0.1086 & 0.9512 & 0.9478 & 0.9478\\
		PSNR 	& 22.444 & 25.838 & 29.110 & 18.094 & 24.135 & 27.689 & 6.5497 & 9.012  & 10.640\\
		SSIM 	& 0.8104 & 0.7487 & 0.8006 & 0.5312 & 0.6551 & 0.7441 & 0.0262 & 0.1193 & 0.0641\\
	\end{tabular}
	\caption{Performances of \pnp, QAB-PnP and P$^4$IP algorithms on gray scale images corrupted by a Gaussian PSF with $\sigma=1$ and $\nu= 20$. \pnp provides better results than QAB-PnP. P$^4$IP instead does not reach reliable results, and suffers particularly from the presence of noise.}
	\label{tab:BW}
\end{table}
\cref{tab:BW} presents the numerical assessment of the performance of the three algorithms.  \cref{fig:BW} shows the recorded data $\vg$ for the three images, together with the recovered images achieved by the three different algorithms. The effect of the PSF is significant, given the images' dimension, and the noise level is rather high. The {restored images} achieved by QAB-PnP present several artifacts, while the ones provided by PnPSplt+ suffer from the loss of details, mainly in \textit{Baby} cases. P$^4$IP failed to recover reliable {images}: the images shown in \cref{fig:BW} related to the result of P$^4$IP are rescaled in order to make them visible, since the reached maximum value is around 0.04 in all three cases.
\begin{figure}[htbp]
	\centering
	\newcommand\factor{0.24}
	\subfigure[$\vg$, \textit{Butterfly}]{\includegraphics[width=\factor\textwidth]{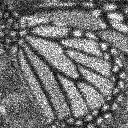}}\hfill\subfigure[\pnp]{\includegraphics[width=\factor\textwidth]{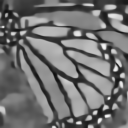}}\hfill\subfigure[QAB]{\includegraphics[width=\factor\textwidth]{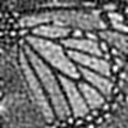}}\hfill\subfigure[P$^4$IP]{\includegraphics[width=\factor\textwidth]{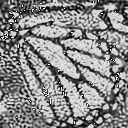}}
	
	\subfigure[$\vg$, \textit{Tucano}]{\includegraphics[width=\factor\textwidth]{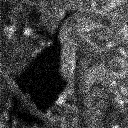}}\hfill\subfigure[\pnp]{\includegraphics[width=\factor\textwidth]{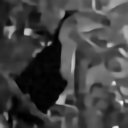}}\hfill\subfigure[QAB]{\includegraphics[width=\factor\textwidth]{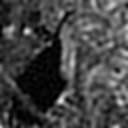}}\hfill\subfigure[P$^4$IP]{\includegraphics[width=\factor\textwidth]{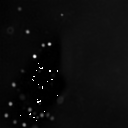}}
	
	\caption{Comparinson on the reconstruction achieved by \pnp, QAB-PNP and P$^4$IP, respectively on the second, third and fourth column. The first column shows the {corrupted} data $\vg$. The results of P$^4$IP are shown in a different scale: while \pnp and QAB provide reconstructions in [0,1], P$^4$IP failed to recover images with values higher than 0.04 in all cases.}
	\label{fig:BW}
\end{figure}

\subsection{Severely Corrupted Images}
The following set of experiment is devoted to assess the performance of the \pnp Algorithm in presence of high noise level or severe blur induced by the PSF. 
\begin{figure}[htbp]
	\centering
	\newcommand{\factor}{0.23}
	\subfigure[$\vg$, $\nu=10$]{\includegraphics[width=\factor\textwidth]{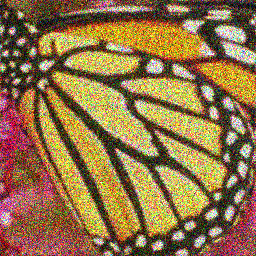}}\hfill\subfigure[Recovered image]{\includegraphics[width=\factor\textwidth]{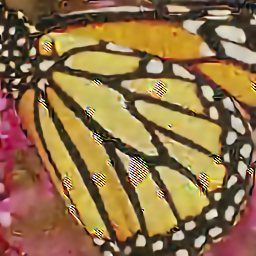}}\hfill
	\subfigure[$\vg$, $\sigma=2.5$]{\includegraphics[width=\factor\textwidth]{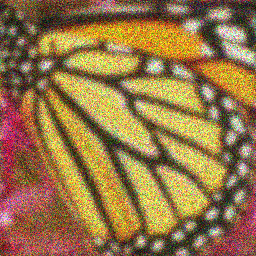}}\hfill\subfigure[Recovered image]{\includegraphics[width=\factor\textwidth]{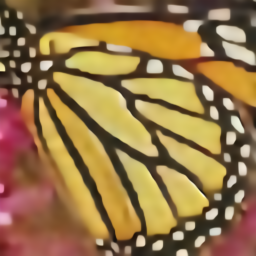}}
	
	\subfigure[$\vg$, $\nu=10$]{\includegraphics[width=\factor\textwidth]{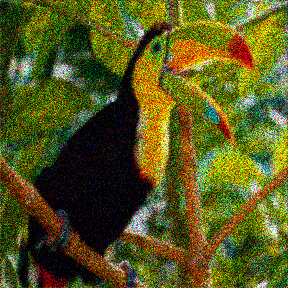}}\hfill
	\subfigure[Recovered image]{\includegraphics[width=\factor\textwidth]{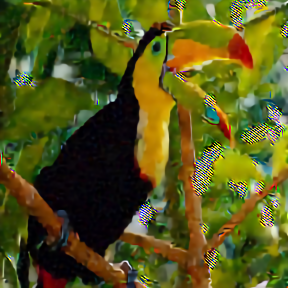}}\hfill\subfigure[$\vg$ $\sigma=2.5$]{\includegraphics[width=\factor\textwidth]{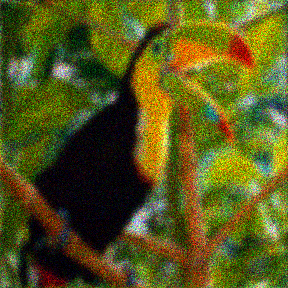}}\hfill\subfigure[Recovered image]{\includegraphics[width=\factor\textwidth]{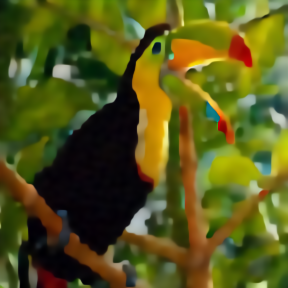}}
	
	\subfigure[$\vg$, $\nu=10$]{\includegraphics[width=\factor\textwidth]{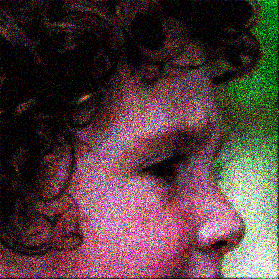}}\hfill\subfigure[Recovered image]{\includegraphics[width=\factor\textwidth]{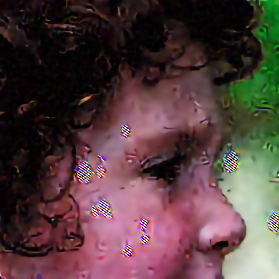}}\hfill\subfigure[$\vg$, $\sigma=2.5$]{\includegraphics[width=\factor\textwidth]{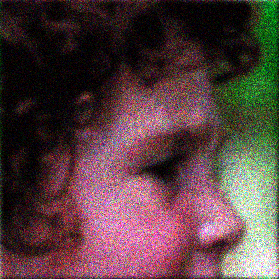}}\hfill\subfigure[Recovered image]{\includegraphics[width=\factor\textwidth]{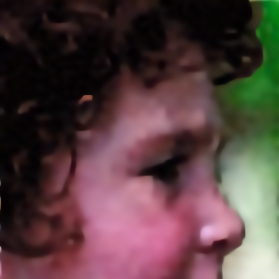}}
	
	\caption{image results when the perturbation on the recorded data is particularly strong, in terms of noise level or blurring. First row: reconstructions obtained for a PSF with $\sigma=1$ and noise level set to 5. Second row: reconstructions obtained for a PSF with $\sigma=2.5$ and noise level set to 20.}\label{fig:highlycorrputed}
\end{figure}
\cref{tab:corrupted} (upper part) presents the numerical performance of \pnp when the Poisson Noise level $\nu$ is increased to 15 and 10. As one expects, the higher the noise level the worst the performances, but nonetheless the achieved results present rather high scores: in particular, the PSNR of the recovered images reaches satisfying levels. \cref{tab:corrupted} (lower part) shows the four scores achieved when  large Gaussian PSF ($\sigma=2$ and $\sigma=2.5$) are used to blur the images, with $\nu=20$. The quality of the restoration is reliable, although in this case the information loss induced by the blurring is too high to retrieve pleasant images to the human eyes.
\begin{table}[htbp]
	\centering
	\begin{tabular}{l|ccc||ccc}
		& \textit{Butterfly} & \textit{Tucano} & \textit{Baby}&\textit{Butterfly} & \textit{Tucano} & \textit{Baby}\\
		\cmidrule{2-7}
		&\multicolumn{3}{|c||}{$\nu=15$}&\multicolumn{3}{|c}{$\nu=10$}\\
		\midrule
		MSE 	& 0.0036 & 0.0016 & 0.0018 & 0.0088 & 0.0096 & 0.0049 \\
		RE 		& 0.1135 & 0.1135 & 0.1110 & 0.1773 & 0.2787 & 0.1818 \\
		PSNR  	& 24.425 & 27.987 & 27.380 & 20.555 & 20.185 & 23.095 \\
		SSIM 	& 0.8369 & 0.8342 & 0.6717 & 0.7533 & 0.7020 & 0.6025 \\
		\midrule\midrule
		&\multicolumn{3}{|c||}{$\sigma=2$}&\multicolumn{3}{|c}{$\sigma=2.5$}\\
		\midrule
		MSE 	& 0.0062 & 0.0021 & 0.0018 & 0.0078 & 0.0028 & 0.0020\\
		RE 		& 0.1493 & 0.1316 & 0.1096 & 0.1671 & 0.1494 & 0.1150\\
		PSNR 	& 22.047 & 26.706 & 27.488 & 21.071 & 25.601 & 27.077\\
		SSIM 	& 0.7588 & 0.7907 & 0.6332 & 0.7172 & 0.7562 & 0.6160\\
		
	\end{tabular}
	\caption{Results achieved by \pnp when the Gaussian PSF induces a larger blur and for high level noise. The loss of information is relatively high, but the index measures are still reliable. }
	\label{tab:corrupted}
\end{table}
\cref{fig:highlycorrputed} presents the recovered images when the noise level is set to 5 and when the PSF inducing the blurring is large ($\sigma=2.5$) for \textit{Butterfly}, \textit{Tucano} and \textit{Baby} images on the first, second and third row, respectively. As one expects, the recovered images present several artefacts, mainly when recovering in presence of high noise, but even in these extreme cases \cref{al:pnpsplit+} manages to recover most of the information.

\subsection{Other Blurring Operators}
\label{ssec:blurs}{
	The proposed \pnp approach is tested on other blur operators, namely 
	\begin{itemize}
		\item Motion blur, with length 11 and slope 35$^\circ$,
		\item Out-of-Focus blur, with radius 5,
		\item Mean Filter blur, with dimension equal to 3.
	\end{itemize}
	The corrupted data $\vg$ is computed using the linear model \eqref{eq:linMod} employing one of the above operators for $\vH$, and the noise level is set again to 20. The parameter $\gamma$ is set to 0.01 and the method is run for 1000 iterations. The corrupted data and the relative restorations are depicted in \cref{fig:otherBlurs}, together with the numerical performance metrics. \pnp provides reliable results under different types of blurring operators.}
\begin{figure}[htbp]
	\centering
	\newcommand{\factor}{0.25}
	\subfigure[Motion Blur]{\includegraphics[width=\factor\textwidth]{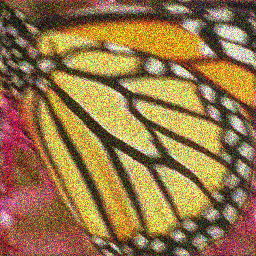}} \hfill \subfigure[Out-of-Focus Blur]{\includegraphics[width=\factor\textwidth]{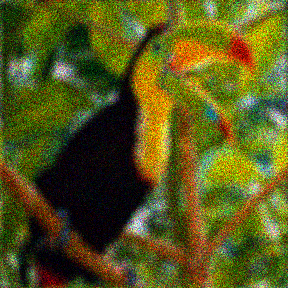}}\hfill\subfigure[Mean Filter Blur]{\includegraphics[width=\factor\textwidth]{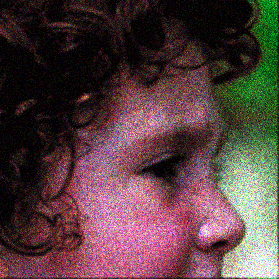}}
	
	\subfigure[PSNR 23.052]{\includegraphics[width=\factor\textwidth]{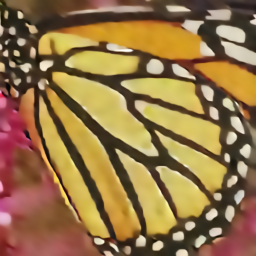}}\hfill \subfigure[PSNR 25.283 ]{\includegraphics[width=\factor\textwidth]{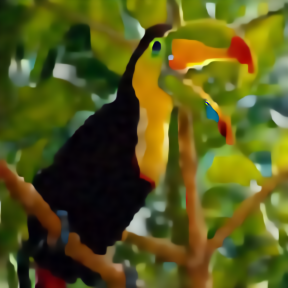}} \hfill\subfigure[PSNR 28.626 ]{\includegraphics[width=\factor\textwidth]{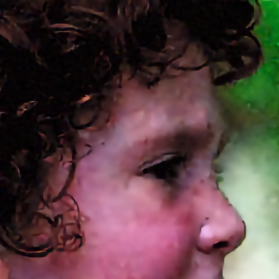}}
	\caption{Result of PnPSplit$^+$ for Motion, Out-of-Focus and Mean filter blurs. Top row: corrupted images; bottom row: recovered images with relative PSNR.}\label{fig:otherBlurs}
\end{figure} 

\subsection{Adaptive strategy for $\gamma$}
\label{ssec:gamma}

{As observed in \cref{rem:adaptiveGamma}, ADMM performance could be particularly dependent on the choice of $\gamma$. The literature presents several adaptive strategies to overcome this issue: this works employs the one depicted in \cite{He2000,Wang2001,Zanni2015},} and \cref{al:pnpsplit+} can be modified inserting the following $\gamma$-scheduler after the update of the Lagrangian parameters.
\begin{equation}
	\label{eq:gammaSchedule}
	\gamma^{k+1} = \begin{cases}
		\ds\frac{\alpha}{\gamma^k} & \text{if } \|\vp^k\|>\mu\|\vd^k\|, \,k\leq k_{max}\\\\
		\ds{\alpha}{\gamma^k} & \text{if } \|\vd^k\|>\mu\|\vp^k\|, \,k\leq k_{max}\\\\
		\gamma^k & \text{otherwise}
	\end{cases}
\end{equation}
where $\alpha$ and $\mu$ are positive values greater than 1. A first glance, it seems that the number of parameters to set rises from one to four: actually, $\alpha$ and $\mu$ can be set really close to 1 and the only parameters to set remain $\gamma^0$ and $k_{max}$. A first experiment is carried out for testing the relevance of the adaptive strategy for $\gamma$, and how the initial parameter influences \cref{al:pnpsplit+} results.

Two images, namely \textit{Butterfly} and \textit{Tucano}, are employed for this test: each one is blurred with a Gaussian PSF with standard deviation $\sigma=1$, and corrupted with Poisson noise at level $\nu=20$. For each choice $\gamma^0$, \cref{al:pnpsplit+} is run also with $\gamma^k=\gamma^0$ for any $k$. The maximum number $K$ of iteration is set to 2500, $\alpha=\mu=1.001$ and $k_{max}=1250$. 
\begin{table}[htbp]
	\centering
	\begin{tabular}{l|cc|cc|cc|cc}
		& \multicolumn{8}{c}{\textit{Butterfly}}\\
		\toprule
		& $\gamma^0:1$ & Adapt & $\gamma^0:10$ & Adapt & $\gamma^0:10^{2}$ &  Adapt & $\gamma^0:10^{3}$ & Adapt \\
		\midrule
		MSE 	& 0.1132 & 0.1129 & 0.1272 & 0.0944 & 0.1293 & 0.0249 & 0.1295 & 0.0035\\
		RE 		& 0.6360 & 0.6349 & 0.6739 & 0.5806 & 0.6795 & 0.2982 & 0.6800 & 0.1121\\
		PSNR	& 9.461  & 9.475  & 8.957  & 10.252 & 8.886  & 16.039 & 8.8879 & 24.538\\
		SSIM	& 0.1686 & 0.1690 & 0.1546 & 0.1902 & 0.1538 & 0.3941 & 0.1528 & 0.8328\\
		\midrule
		& $\gamma^0:10^{-4}$ & Adapt & $\gamma^0:10^{-3}$ & Adapt & $\gamma^0:10^{-2}$ &  Adapt & $\gamma^0:10^{-1}$ & Adapt \\
		\midrule
		MSE 	& 0.0150 & 0.0127 & 0.0086 & 0.0047 & 0.0031 & 0.0072 & 0.0755 & 0.0871\\
		RE 		& 0.2312 & 0.2132 & 0.1754 & 0.1295 & 0.1046 & 0.1604 & 0.5193 & 0.5577\\
		PSNR	& 18.248 & 18.953 & 20.648 & 23.284 & 25.137 & 21.425 & 11.220 & 10.600\\
		SSIM	& 0.6302 & 0.6609 & 0.7267 & 0.8024 & 0.8526 & 0.6629 & 0.2263 & 0.2026\\
		\midrule
		\midrule
		& \multicolumn{8}{c}{\textit{Tucano}}\\
		\toprule
		
		& $\gamma^0:10^{-4}$ & Adapt & $\gamma^0:10^{-3}$ & Adapt & $\gamma^0:10^{-2}$ &  Adapt & $\gamma^0:10^{-1}$ & Adapt \\
		\midrule
		MSE 	& 0.0098 & 0.0078 & 0.0047 & 0.0020 & 0.0013 & 0.0048 & 0.0577 & 0.0693\\
		RE 		& 0.2812 & 0.2507 & 0.1947 & 0.1283 & 0.1017 & 0.1971 & 0.6842 & 0.7497\\
		PSNR	& 20.108 & 21.104 & 23.302 & 26.921 & 28.946 & 23.195 & 12.385 & 11.591\\
		SSIM	& 0.5378 & 0.6021 & 0.7043 & 0.8147 & 0.8582 & 0.6531 & 0.2509 & 0.1894\\
		\midrule
		& $\gamma^0:1$ & Adapt & $\gamma^0:10$ & Adapt & $\gamma^0:10^{2}$ &  Adapt & $\gamma^0:10^{3}$ & Adapt \\
		\midrule
		MSE 	& 0.0859 & 0.0884 & 0.0986 & 0.0726 & 0.1001 & 0.0160 & 0.1003 & 0.0014\\
		RE 		& 0.8344 & 0.8465 & 0.8939 & 0.7669 & 0.9008 & 0.3602 & 0.9016 & 0.1079\\
		PSNR	& 10.660 & 10.536 & 10.063 & 11.393 & 9.996  & 17.957 & 9.988  & 28.426\\
		SSIM	& 0.1424 & 0.1375 & 0.1272 & 0.1618 & 0.1411 & 0.4119 & 0.1244 & 0.8478\\

	\end{tabular}
	\caption{Evaluation of fixed versus adaptive strategy. The column $\gamma^0$ denotes the value for $\gamma$ selected as initial one for the adaptive strategy (Adapt column) and the constant used in the vanilla \pnp. The index measures of Mean Square Error, Relative Error, Peak Signal to Noise Ratio and Similarity Structure Index Measure are employed for the comparison. The adaptive strategy is particularly effective for high values of $\gamma$.} 
	\label{tab:adaptiveComp}
\end{table}
\cref{tab:adaptiveComp} shows the comparison result. {Considering a constant value for $\gamma$, the best choice is $\gamma=0.01$, which is exactly the parameter used for the noisy images employed for the training of $\den$ used in \cref{al:pnpsplit+}. One should note that the adaptive strategy when $\gamma^0=0.01$ fails, while in most all the other cases such strategy improves the quality of the restored images. These tests confirm that choosing the constant value $\gamma=0.01$ provides reliable performances under the usage of the denoiser presented in \cite{doi:10.1137/20M1387961}. If one employs a denoiser whose training is carried out on noisy images with different noise levels, then using the adaptive strategy could overcome possible issues: in \cref{ssec:nonconv} this approach shall reveal to provide remarkable results.}

\subsection{Performance without Convergence Guarantees}
\label{ssec:nonconv}
The last runs of experiments consists of the implementation of \cref{al:pnpsplit+} when the denoising network is not firmly non expansive, \emph{i.e.} not abiding to the hypothesis that guarantees the convergence of the method. The network trained in \cite{doi:10.1137/20M1387961} is substituted by the classical deep convolutional  network presented in \cite{7839189}. Such network has been trained by minimizing the well--known MSE loss function, therefore without imposing any constraint that forces the non firmly expansiveness. 
\begin{figure}[htbp]
	\centering
	\newcommand\factor{0.25}
	\subfigure[\textit{Butterfly}]{\includegraphics[width=\factor\textwidth]{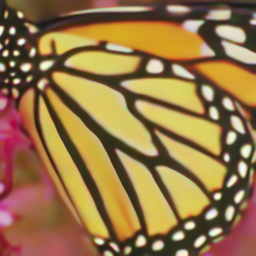}}\hfill\subfigure[\textit{Tucano}]{\includegraphics[width=\factor\textwidth]{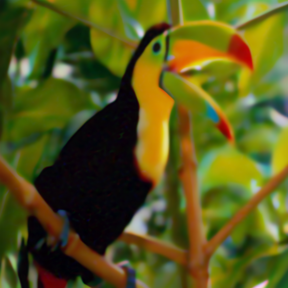}}\hfill\subfigure[\textit{Baby}]{\includegraphics[width=\factor\textwidth]{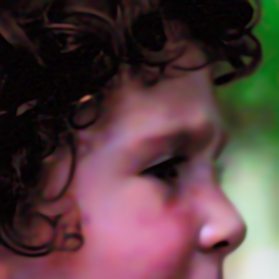}}
	\caption{Recovered images when the convergence guarantees are not met. The quality of these reconstructions is similar to the quality of the images obtained emploing a net satisfying the convergence guarantees, both in terms of visual inspection and performance measures.}
	\label{fig:noConvergence}
\end{figure}
The results are collected in \cref{fig:noConvergence} and the numerical performance is summed up in \cref{tab:noConvergence}, for the case in which the blur is induced by a Gaussian PSF with $\sigma=1$ and the noise level is set to 20. {Setting a fixed value for $\gamma$ does not yield to reliable results, likely due to the absence of convergence property. Therefore, these tests are run employing the adaptive strategy depicted in \cref{ssec:gamma},  setting  the initial value of $\gamma$ to 10, the maximum number of iterations to 1000 and $k_{max}=1000$ in \eqref{eq:gammaSchedule}}.
\begin{table}[htbp]
	\centering
	\begin{tabular}{l|ccc}
		& \textit{Butterfly} & \textit{Tucano} & \textit{Baby}\\
		\midrule
		MSE & 0.0050 & 0.0015 & 0.0015 \\
		RE & 0.1337 & 0.1113 & 0.1017 \\
		PSNR & 23.003 & 28.163 & 28.137 \\
		SSIM & 0.7894 & 0.8267 & 0.6528 \\
	\end{tabular}
	\caption{Numerical assessment of the {achieved results} when a net not satisfying the requirements of non {firmly} expansion is not met. The indexes values are slightly lower than the ones obtained in \cref{ssec:comparison}: this could be due to the denoising network.}\label{tab:noConvergence}
\end{table} The numerical experience shows that even employing a network that, at a first glance, does not assure the convergence of the method allows to achieve reliable results, both in terms of visual inspection and of measurement indexes. The latter ones do not achieve values close to the ones in \cref{tab:comparison}: this could be due to the quality of denoising ability of the network. 

\section{Conclusion}
\label{sec:conclusions}
This work presented a novel approach, named \pnp, for solving image restoration problems in presence of Poisson noise. The original idea of \cite{Setzer2012} is coupled with PnP strategy of substituting the proximal step on the regularization function with an off--the--shelf denoiser. In particular, for ensuring the convergence of the method a firmly non expansive denoiser has been employed in the \pnp scheme. The main contribution of this approach is to avoid the usage of an inner solver for the deblurring step, allowing the computation of solution to the inner problem via an explicit formula. This strategy showed remarkable performances, both in terms of quality measurements and computational time, in comparison to state of the art algorithms, and even in presence of high noise levels and when the blurring effect of the PSF is significant.

The results are really promising, but nonetheless there are still several aspects to explore and improve. On the one hand, {The adaptive strategy should be further investigated for both firmly non-expansive denoisers and those that do not satisfy convergence guarantees. In the latter case, the adaptive strategy may play a role in the convergence behavior of the method. Another aspect to be considered in future research is the employment of firmly non expansive blind denoising networks, that is networks trained for the noise removal with different standard deviations. A further research direction could consider the development of a suitable stopping criterion, aiming to avoid early stopping via the setting the of maximum number of iteration.} Finally, a further generalization of the proposed approach can be done in the direction of Proximal Gradient Descent Ascent methods.

\paragraph{Acknowledgments.}This research has been partially performed in the framework of the MIUR-PRIN Grant 20225STXSB "Sustainable Tomographic Imaging with Learning and Regularization", GNCS Project CUP E53C24001950001 "Metodi avanzati di ottimizzazione stocastica per la risoluzione di problemi inversi di imaging", and CARIPLO project "Project Data Science Approach for Carbon Farming Scenarios (DaSACaF)"CAR\_RIC25ABENF\_01.

\bibliographystyle{plain}
\bibliography{PoissPnP}
\end{document}